\documentclass[12pt]{amsart}
\usepackage[mathscr]{eucal}
\usepackage{latexsym, amssymb, amsthm, amsmath}
\usepackage{epsfig, xspace}
\usepackage[usenames, dvips]{color}
\usepackage{ifthen}

\newcommand{\CC}{\mathbb{C}}
\newcommand{\FF}{\mathbb{F}}

\newcommand{\NN}{\mathbb{N}}

\newcommand{\RR}{\mathbb{R}}
\newcommand{\CP}{\mathbb{CP}}

\newcommand{\RP}{\mathbb{RP}}
\newcommand{\PP}{\mathbb{P}}

\newcommand{\polm}{\CC_{m-1}[z]}
\newcommand{\Rpolm}{\RR_{m-1}[z]}
\newcommand{\pol}[1]{\CC_{#1}[z]}

\newcommand{\Wr}{\mathrm{Wr}}
\newcommand{\Gr}{\mathrm{Gr}}
\newcommand{\OGpol}{\mathrm{OG}(n,\pol{2n})}
\newcommand{\OGn}{\mathrm{OG}(n,2n{+}1)}
\newcommand{\derz}{{\tfrac{d}{dz}}}

\newcommand{\bfdef}[1]{{\bf \emph{#1}}}

\newtheorem{lemma}{Lemma}
\newtheorem{theorem}[lemma]{Theorem}

\newtheorem{proposition}[lemma]{Proposition}

\title{Reality and transversality for Schubert calculus 
in $\OGn$}
\author{Kevin Purbhoo}
\address{Department of Combinatorics \& Optimization \\
         University of Waterloo \\
         Waterloo, ON, N2L 3G1\\
         CANADA}
\email{kpurbhoo@math.uwaterloo.ca}
\urladdr{http://www.math.uwaterloo.ca/\~{}kpurbhoo}
\thanks{Research partially supported by an NSERC discovery grant.}

\begin{document}

\begin{abstract}
We prove an analogue of the Mukhin-Tarasov-Varchenko theorem 
(formerly
the Shapiro-Shapiro conjecture) for the maximal
type $B_n$ orthogonal Grassmannian $\OGn$.
\end{abstract}

\maketitle
%%%%%%%%%%%%%%%%%%%%%%%%%%%%%%%%%%%%%%%%%%%%%%%%%%%%%%%%%%%%%%%%%%%%%%

\section{The Mukhin-Tarasov-Varchenko Theorem}
\label{sec:MTV}

%%%%%%%%%%%%%%%%%%%%%%%%%%%%%%%%%%%%%%%%%%%%%%%%%%%%%%%%%%%%%%%%%%%%%%
For any non-negative integer $k$,
let $\pol{k}$ denote the $(k{+}1)$-dimensional
complex vector space of polynomials of degree at most $k$:
$$\pol{k} := \{f(z) \in \FF[z] \mid \deg f(z) \leq k\}\,.$$
Fix integers $0 \leq d \leq m$, and consider the Grassmannian
$X = \Gr(d,\polm)$, the variety of all $d$-dimensional linear
subspaces of the $m$-dimensional vector space $\polm$.
A point $x \in X$ is \bfdef{real}
if $x$ is is spanned by polynomials in $\Rpolm$; a subset of $S \subset X$ 
is real if every point in $S$ is real.

The Mukhin-Tarasov-Varchenko theorem (formerly the Shapiro-Shapiro
conjecture) asserts that any zero-dimensional intersection of Schubert 
varieties in $X$, relative a special family of flags in $\polm$,
is transverse and real.
This theorem is remarkable for two immediate reasons: first,
it is a rare example of
an algebraic geometry problem in which the solutions are always
provably real; second, the usual arguments to
prove transversality involve Kleiman's transversality theorem~\cite{Kle},
which requires that the Schubert
varieties be defined relative to generic flags.
We recall the most relevant statements here, and
refer the reader to the survey article~\cite{Sot-F}
for a discussion of the history, context, reformulations and
applications of this theorem.

To begin, we define a full flag in $\polm$, for each $a \in \CP^1$: 
$$F_\bullet(a) 
\ :\ \{0\} \subset F_1(a) \subset \dots \subset F_{m-1}(a) \subset \polm\,.$$
If $a \in \CC$, $$F_i(a) := (z+a)^{m-i}\CC[z] \cap \polm$$
is the set of
polynomials in $\polm$ divisible by $(z+a)^{m-i}$.
For $a = \infty$, we set
$F_i(\infty) := \pol{i-1} = \lim_{a \to \infty} F_i(a)$.
The flag $F_\bullet(a)$ is often described as the flag osculating the
rational normal curve $\gamma : \CP^1 \to \PP(\polm)$, 
$\gamma(t) = (z+t)^{m-1}$, 
which simply means that $F_i(a)$ is the span of 
$\{\gamma(a), \gamma'(a), \dots, \gamma^{(i-1)}(a)\}$.

Let $\Lambda = \Lambda_{d,m}$ be the set of all partitions 
$\lambda : (\lambda^1 \geq \dots \geq \lambda^d)$, where
$\lambda^1 \leq m-d$ and $\lambda^d \geq 0$.  We say $\lambda$
is a partition of $k$ and write
$\lambda \vdash k$ or $|\lambda| = k$ if
$k = \lambda^1+ \dots + \lambda^d$.
For every $\lambda \in \Lambda$, the \bfdef{Schubert Variety} in $X$
relative to the flag $F_\bullet(a)$ is
$$X_\lambda(a) 
:= \{x \in X \mid \dim \big(x \cap F_{n-d-\lambda^i+i}(a) \big) \geq i\,,
\text{ for $i=1, \dots, d$}\}\,.$$
The codimension of $X_\lambda(a)$ in $X$ is $|\lambda|$.   

\begin{theorem}[Mukhin-Tarasov-Varchenko~\cite{MTV1, MTV2}]
\label{thm:MTV}
If $a_1, \dots a_s \in \RP^1$ are distinct real points, and
$\lambda_1, \dots \lambda_s \in \Lambda$ are partitions with
$|\lambda_1| + \dots + |\lambda_s| = \dim X$, then the
intersection
$$X_{\lambda_1}(a_1) \cap \dots \cap X_{\lambda_s}(a_s)$$
is finite, transverse, and real.
\end{theorem}

In~\cite{Sot-F}, Sottile conjectured an
analogue of Theorem~\ref{thm:MTV} for $\OGn$,
the maximal orthogonal Grassmannian in type $B_n$.
In Section~\ref{sec:OG} of this note, we give a proof of this
conjecture (our Theorem~\ref{thm:maintheorem}).  
We discuss some of its consequences in Section~\ref{sec:more}; 
in particular, we
note that Theorem~\ref{thm:maintheorem} should yield a 
geometric proof of the Littlewood-Richardson rule for~$\OGn$.

%%%%%%%%%%%%%%%%%%%%%%%%%%%%%%%%%%%%%%%%%%%%%%%%%%%%%%%%%%%%%%%%%%%%%%

\section{The theorem for $\OGn$}
\label{sec:OG}

%%%%%%%%%%%%%%%%%%%%%%%%%%%%%%%%%%%%%%%%%%%%%%%%%%%%%%%%%%%%%%%%%%%%%%

Fix a positive integer $n$, and consider
the non-degenerate symmetric bilinear form
$\langle \cdot, \cdot \rangle$ on the $(2n+1)$-dimensional vector space
$\pol{2n}$ given by
$$\Big\langle \sum_{k=0}^{2n} a_k \tfrac{z^k}{k!}\,,\,
\sum_{\ell=0}^{2n} b_\ell \tfrac{z^\ell}{\ell!} \Big\rangle
 = \sum_{m =0}^{2n} (-1)^m a_m b_{2n-m}\,.$$
Let $Y = \OGpol)$ be the orthogonal Grassmannian in
$\pol{2n}$, which is the variety of all $n$-dimensional isotropic 
subspaces of $\pol{2n}$.  
The dimension of $Y$ is $\frac{n(n+1)}{2}$.

The definition of a Schubert variety in $Y$
requires our reference flags to be orthogonal.
The bilinear form on $\pol{2n}$ has been chosen so that this is
true for the flags $F_\bullet(a)$.

\begin{proposition}
For $a \in \CP^1$, then the flag $F_\bullet(a)$ is an orthogonal flag;
that is $F_i(a)^\perp = F_{2n+1-i}(a)$, for $i=0, \dots, 2n+1$.
\end{proposition}

\begin{proof}
For $a=0, \infty$, this is straightforward to verify.  We deduce
the result for all other $a$ by showing that
$\langle f(z), g(z) \rangle = \langle f(z+a), g(z+a)\rangle$.

To see this, note that
$\langle \derz(\tfrac{z^k}{k!}) , \tfrac{z^\ell}{\ell!} \rangle
%= (-1)^\ell\delta_{k-1,2n-\ell} 
%= (-1)^\ell\delta_{2n-k,\ell-1} 
= - \langle \tfrac{z^k}{k!} , \derz(\tfrac{z^\ell}{\ell!}) \rangle$,
so $\derz$ is a skew-symmetric operator 
on $\pol{2n}$.  It follows that
$\exp(a \frac{d}{dz})$ is an orthogonal operator
on $\pol{2n}$ and so
$\langle f(z+a), g(z+a)\rangle =
\langle \exp(a \derz)f(z), \exp(a \derz)g(z)\rangle =
\langle f(z), g(z) \rangle$.
\end{proof}

The Schubert varieties in $Y$ are indexed by the set
$\Sigma$ of all {\em strict partitions}
$\sigma :  (\sigma^1> \sigma^2 > \dots > \sigma^k)$,  with
$\sigma^1 \leq n$, $\sigma^k > 0$, $k \leq n$. 
For convenience, we put 
$\sigma^j = 0$ for $j  > k$.  We associate to $\sigma$
a decreasing sequence of integers,
$\overline \sigma^1 >  \dots > \overline \sigma^n$,
such that $\overline \sigma^i = \sigma^i$ if
$\sigma^i > 0$, and 
$\{|\overline \sigma^1|, \dots, |\overline \sigma^n|\} 
= \{1, \dots,n\}$.
It is not hard to see that $\overline \sigma^i$ is given explicitly
by the formula
$$\overline \sigma^i = 
\sigma^i-i + \#\{j \in \NN \mid j \leq i < j+\sigma^j\}\,.$$
For $\sigma \in \Sigma$, the \bfdef{Schubert variety} in $Y$ 
relative to the flag $F_\bullet(a)$ is defined to be
$$Y_\sigma(a) 
:= \{y \in Y \mid 
\dim \big(y \cap F_{1+n-\overline \sigma^i}(a)\big) \geq i\,,
\text{ for $i=1, \dots, n$}\}\,.$$
The codimension of $Y_\sigma(a)$ in $Y$ is $|\sigma|$.
We refer the reader to~\cite{FP, Sot-OG} for further details.

\begin{theorem}
\label{thm:maintheorem}
If $a_1, \dots a_s \in \RP^1$ are distinct real points, and
$\sigma_1, \dots \sigma_s \in \Sigma$, with  
$|\sigma_1| + \dots + |\sigma_s| = \dim Y$, then the
intersection
$$Y_{\sigma_1}(a_1) \cap \dots \cap Y_{\sigma_s}(a_s)$$
is finite, transverse, and real.
\end{theorem}

\begin{proof}
Let $X = \Gr(n,\pol{2n})$, 
and let $\Lambda = \Lambda_{n,2n+1}$.
We prove this result by viewing $Y$ as a subvariety
of $X$, and the Schubert varieties $Y_\sigma$ as the intersections
of Schubert varieties in $X$ with $Y$.  Note that 
$\dim X = 2\dim Y = n(n+1)$.

For a strict partition
$\sigma \in \Sigma$, let
$$
\widetilde \sigma^i := \overline \sigma^i + i 
= \sigma^i + \#\{j \in \NN \mid j \leq i < j+\sigma^j\} \,.
$$
Observe that 
$\widetilde \sigma^i - \widetilde \sigma^{i+1} = 
\overline\sigma^i - \overline \sigma^{i+1} -1 \geq 0$,
%\begin{align*}
%\widetilde \sigma^{i+1} - \widetilde \sigma^i 
%&= \sigma^{i+1}- \sigma^i -
%\#\Big (\big\{j \ \big|\ j \leq i < j{+}\sigma^j\bigr\} \setminus
%%\bigl\{j \ \big|\ j \leq i{+}1 < j{+}\sigma^j\bigr\} \Big)\\
%&\geq 1 - 1 = 0\,,
%\end{align*}
and
$\widetilde \sigma^1 \leq \sigma^1+1 \leq n+1$;
hence we see that
$$\widetilde \sigma \ :\ (\widetilde \sigma^1 \geq \widetilde \sigma^2 \geq 
\dots \geq 
\widetilde \sigma^n)$$ 
is a partition in $\Lambda$.

It follows directly from the
definitions of Schubert varieties in $X$ and $Y$ that
$$X_{\widetilde\sigma}(a) \cap Y = Y_\sigma(a)\,.$$
Moreover, we have,
\begin{align*}
|\widetilde \sigma| &= 
|\sigma| + \sum_{i \geq 1}
\#\{j \in \NN \mid j \leq i < j+\sigma^j\}\\
&= |\sigma| + \sum_{j \geq 1} \#\{i \in \NN \mid j \leq i < j+\sigma^j\} \\
&= |\sigma| + \sum_{j \geq 1} \sigma^j 
\ =\ 2|\sigma|\,.
\end{align*}

Thus, if
$|\sigma_1| + \dots + |\sigma_s| = \dim Y$, then
$|\widetilde \sigma_1| + \dots + |\widetilde \sigma_s| = 2\dim Y = \dim X$,
and so by Theorem~\ref{thm:MTV} the intersection 
$$X_{\widetilde \sigma_1}(a_1) \cap \dots \cap X_{\widetilde \sigma_s}(a_s)$$
is finite, transverse, and real;
in particular this intersection is a zero-dimensional reduced scheme.
It follows immediately that
$$Y_{\sigma_1}(a_1) \cap \dots \cap Y_{\sigma_s}(a_s) =
Y \cap X_{\widetilde \sigma_1}(a_1) \cap \dots 
\cap X_{\widetilde \sigma_s}(a_s)$$ 
is finite and real.  To see that the intersection on the left hand side
is also transverse, note that it is proper, so it suffices to
show that it is scheme-theoretically reduced.  But this is immediate
from the fact that the right hand side
is the intersection of $Y$ with a zero-dimensional reduced scheme.
\end{proof}

%%%%%%%%%%%%%%%%%%%%%%%%%%%%%%%%%%%%%%%%%%%%%%%%%%%%%%%%%%%%%%%%%%%%%%

\section{Consequences}
\label{sec:more}

%%%%%%%%%%%%%%%%%%%%%%%%%%%%%%%%%%%%%%%%%%%%%%%%%%%%%%%%%%%%%%%%%%%%%%

Let $0 \leq d \leq m$, $X = \Gr(d, \polm)$, be as 
in Section~\ref{sec:MTV}.
We can consider the Wronskian
of $d$ polynomials $f_1(z), \dots, f_d(z) \in \polm$:
$$\Wr_{f_1, \dots, f_d}(z) :=
\begin{vmatrix}
f_1(z) & \cdots & f_d(z)\\
f_1'(z) & \cdots  & f_d'(z) \\
\vdots &  \vdots & \vdots \\
f_1^{(d-1)}(z) & \cdots & f_d^{(d-1)}(z)
\end{vmatrix} \,.$$
This is a polynomial of degree at most $\dim X = d(n-d)$.
If $f_1, \dots, f_d$ are
linearly dependent, the Wronskian is zero; otherwise up to a constant
multiple, $\Wr_{f_1, \dots, f_d}(z)$ depends only on the linear span
of $f_1(z), \dots, f_d(z)$ in  $\polm$.  Thus the
Wronskian gives us a well defined morphism of schemes
$\Wr: X \to \PP(\pol{d(n-d)})$, called the \bfdef{Wronski map}.  
This morphism is flat and finite~\cite{EH}.
For
$x \in X$ we will write $\Wr(x;z)$ for any representative of $\Wr(x)$ in
$\pol{d(n-d)}$.

The Wronski map has a deep connection to the Schubert varieties 
on $X$ relative to the flags $F_\bullet(a)$, $a \in \CP^1$.  A proof
of the following classical result may be found in~\cite{EH,Pur,Sot-F}.
\begin{theorem}
\label{thm:Wrschubert}
The Wronksian $\Wr(x;z)$ is divisible by 
$(z+a)^k$ if and only if $x \in X_\lambda(a)$ for some partition
$\lambda \vdash k$.  Also, $x \in X_\mu(\infty)$ for
some $\mu \vdash \big(\dim X - \deg \Wr(x;z)\big)$.
\end{theorem}

For $X = \Gr(n,\pol{2n})$, and $Y = \OGpol$ we deduce the
following analogue:
\begin{theorem}
\label{thm:Pschubert}
If $y \in Y$ then $\Wr(y;z) = P(y;z)^2$ for some polynomial 
$P(y;z) \in \pol{n(n+1)/2}$.  
$P(y;z)$ is divisible by $(z+a)^k$ if and
only if $y \in Y_\sigma(a)$ for some strict partition 
$\sigma \vdash k$ in $\Sigma$.  Also, $y \in Y_\tau(\infty)$ for
some strict partition $\tau \vdash \big(\dim Y - \deg P(y;z)\big)$.
\end{theorem}

\begin{proof}
Let $y \in Y$, and let $(z+a)^\ell$ be the largest power $(z+a)$ 
that divides $\Wr(x;z)$.  By Theorem~\ref{thm:Wrschubert},
there exists a partition $\lambda \vdash \ell$ 
such that $y \in X_\lambda(a)$.  Since $\ell$ is maximal, $y$ 
is in the Schubert cell
\begin{align*}
X^\circ_\lambda(a) 
&:= \big\{x \in X \ \big|\ \dim \big(x \cap F_k(a) 
\big) \geq i,\ %
n{+}1{-}\lambda^i{+}i \leq k \leq n{+}1{-}\lambda^{i+1}{+}i,\,
0 \leq i \leq n\big\} \\
&= X_\lambda(a) \setminus \bigg(
\bigcup_{|\mu| > |\lambda|} X_\mu(a)
\bigg)\,.
\end{align*}
(Here, by convention, $\lambda^0 = n+1$, $\lambda^{n+1} =0$.)
The Schubert cells in $Y$ are of the form
\begin{align*}
Y^\circ_\sigma(a) 
&:= \big\{y \in Y \ \big|\ \dim \big(y \cap F_k(a) \big) \geq i,\ %
n{+}1{-}\overline \sigma^i \leq k \leq n{-}\overline \sigma^{i+1},\,
0 \leq i \leq n \big\}  \\
&= X_{\widetilde \sigma}(a) \cap Y
\end{align*}
(Here, by convention, $\overline \sigma^0 = n+1$, 
$\overline \sigma^{n+1} = -n-1$.)
Now, the intersection $X^\circ_\lambda(a) \cap Y$ is nonempty, 
since it contains $y$,
and is therefore a Schubert cell in $Y$.
%But the Schubert cells 
%in $Y$ are of the form
%$$Y^\circ_\sigma(a) = \
%
%$Y^\circ_\sigma(a) = X^\circ_{\widetilde \sigma}(a) \cap Y$;
It follows that $\lambda = \widetilde \kappa$ for some strict
partition
$\kappa \in \Sigma$.  Thus $\ell = |\lambda| = 2|\kappa|$ is even,
which proves that $\Wr(y;z) = P(y;z)^2$ is a square.

We have shown that $(z+a)^{|\kappa|}$ is the largest
power of $(z+a)$ that divides $P(y;z)$, and 
$y \in Y^\circ_\kappa(a)$.
If $y \in Y_\sigma(a)$ then we must have $Y_\sigma(a) 
\supset Y_\kappa(a)$, which implies that $|\sigma| \leq |\kappa|$, and
hence $(z+a)^k$ divides $P(y;z)$.  Conversely, for any
$k \leq |\kappa|$ there exists $\sigma \vdash k$ such that 
$Y_\sigma(a) \supset Y_\kappa(a)$, and so $y \in Y_\sigma(a)$.  
This proves the second assertion.  
The third is proved by the same argument, taking 
$\ell = \dim Y - \deg P(y;z)$ and $a = \infty$.
\end{proof}

If we write $P(y)$ for the class of $P(y;z)$ in projective space
$\PP(\pol{n(n+1)/2})$, then $y \mapsto P(y)$ defines a morphism
of schemes $P : Y \to \PP(\pol{n(n+1)/2})$.

\begin{theorem}
\label{thm:Pflatfinite}
$P$ is a flat, finite morphism.
\end{theorem}

\begin{proof}
Let $h(z) = (z+a_1)^{k_1} \dotsm (z+a_s)^{k_s} \in \pol{n(n+1)/2}$.
By Theorem~\ref{thm:Pschubert},
$$P^{-1}(h(z)) = \bigcap_{i=1}^s 
\bigg(\bigcup_{\sigma_i \vdash k_i} Y_{\sigma_i}(a_i)\bigg)\,,$$
which, by Theorem~\ref{thm:maintheorem}, is a finite set.  
Since $P$ is
a projective morphism, this implies that that $P$ is flat
and finite~\cite[Ch. III, Exer. 9.3(a)]{Har}.
\end{proof}

In \cite{Pur} we showed that the properties of the Wronski map
and Theorem~\ref{thm:MTV} can be used to give geometric 
interpretations and proofs of
several combinatorial theorems in the jeu-de-taquin theory,
including the Littlewood-Richardson rule for Grassmannians in 
type $A_n$.  The map $P$ and Theorem~\ref{thm:maintheorem} are the 
appropriate analogues for $\OGn$. With a few modifications, 
it should be possible to use the arguments in~\cite{Pur}
to give geometric proofs of the analogous results in the theory of
shifted tableaux, as developed in~\cite{Hai, Pra, Sag, Ste, Wor}, 
including the Littlewood-Richardson rule for $\OGn$.
The main ingredients required to adapt these proofs are
Theorems~\ref{thm:maintheorem},~\ref{thm:Pschubert} 
and~\ref{thm:Pflatfinite},
and the Gel'fand-Tsetlin toric degeneration of $\OGn$.
The complete details should be straightforward but somewhat 
lengthy, and we will not include them here.

%%%%%%%%%%%%%%%%%%%%%%%%%%%%%%%%%%%%%%%%%%%%%%%%%%%%%%%%%%%%%%%%%%%%%%

\end{document}